\documentclass{article}

\usepackage[preprint]{cpal_2024}

\makeatletter
\newcommand*{\inlineequation}[2][]{%
  \begingroup
    \refstepcounter{equation}%
    \ifx\\#1\\%
    \else
      \label{#1}%
    \fi
    \relpenalty=10000 %
    \binoppenalty=10000 %
    \ensuremath{%
      #2%
    }%
    ~\@eqnnum
  \endgroup
}
\makeatother

\usepackage{xcolor}

\usepackage{algorithm}
\usepackage{algpseudocode}

\usepackage{amsmath}
\usepackage{amssymb}
\usepackage{amsthm}

\newtheorem{theorem}{Theorem}
\newtheorem{lemma}{Lemma}
\newtheorem{prop}{Proposition}

\usepackage{graphicx}
\usepackage{caption}
\usepackage{subcaption}
\usepackage{float}

\usepackage{url}

\title{Convergence analysis of stochastic gradient descent with adaptive preconditioning for non-convex and convex functions}

\author{%
Dmitrii A. Pasechnyuk\textsuperscript{1,2,3,4}, ~Alexander Gasnikov\textsuperscript{2,3,4}, ~Martin Tak\'{a}\v{c}\textsuperscript{1}\\
  \textsuperscript{1}Mohamed bin Zayed University of Artificial Intelligence, UAE,\\
  \textsuperscript{2}Moscow Institute of Physics and Technology, Russia\\
  \textsuperscript{3}Kharkevich Institute for Information Transmission Problems RAS, Russia\\
  \textsuperscript{4}Ivannikov Institute for System Programming RAS, Russia\\
  \texttt{dmivilensky1@gmail.com, gasnikov@yandex.ru, Takac.MT@gmail.com}
}

\begin{document}

\maketitle

\begin{abstract}

Preconditioning is a crucial operation in gradient-based numerical optimisation. It helps decrease the local condition number of a function by appropriately transforming its gradient. For a convex function, where the gradient can be computed exactly, the optimal linear transformation corresponds to the inverse of the Hessian operator, while the optimal convex transformation is the convex conjugate of the function. 
Different conditions result in variations of these dependencies. Practical algorithms often employ low-rank or stochastic approximations of the inverse Hessian matrix for preconditioning. However, theoretical guarantees for these algorithms typically lack a justification for the defining property of preconditioning. 
    This paper presents a simple theoretical framework that demonstrates, given a smooth function and an available unbiased stochastic approximation of its gradient, that it is possible to refine the dependency of the convergence rate on the Lipschitz constant of the gradient.
    
\end{abstract}

\section{Introduction}
    Although a significant portion of the history of preconditioning was dedicated to solving systems of linear equations and quadratic programming, we direct the reader to \cite{benzi2002preconditioning,qu2022optimal} and the references therein instead of discussing it here. In general optimisation which considers problems of the following form:
    \begin{equation} \label{eq:problem_statement}
        \min_{x \in Q} f(x),
    \end{equation}
    preconditioning was one of the first methods developed to generalise the gradient descent method. It has multiple interpretations: change of variables, modification of the function's model to minimize at the current step, and aligning the step direction closer to the Newton one \cite{boyd2004convex}. The latter interpretation is the primary inspiration for this paper.
Let us introduce a simple, yet general, update formula for a preconditioned method, without a detailed description of the context:
    \begin{equation}
        x_{t+1} \leftarrow x_t - \alpha_t k_t(g_t).
    \end{equation}
    Unless stated otherwise, $P_t$ is defined by $k_t(g) = P_t^{-1} g$, i.e., linear preconditioning is considered, and the objective function $f$ is assumed to be convex. Preconditioning is closely related to Newton-type methods, which are preconditioned methods with $P_t = \nabla^2 f(x_t)$. If $g_t = \nabla f(x_t)$, the (damped) Newton method converges (asymptotically) as fast as the optimal gradient method \cite{hanzely2022damped} with a proper choice of step sizes $\alpha_t$, and super-linearly in the vicinity of the solution \cite{polyak1987introduction}. If $g_t = \nabla f(x_t) + \xi_t$, where $\xi_t$ is a random variable with a known probability density function $p$, the convex preconditioned method with $k_t(g) = (\nabla^2 f(x^*))^{-1} (J(p))^{-1} \nabla (\log p) (g)$, where $J(p)$ is the Fisher information matrix for $p$, is asymptotically optimal among convex preconditioned methods \cite{polyak1980optimal}. Despite the important role of the Fisher information matrix in stochastic preconditioned methods \cite{li2018preconditioner}, it is beyond the scope of this paper. It is evident that methods involving $(\nabla^2 f)^{-1}$ in the preconditioning operation are worth studying.

    Inverse Hessian computation is a time- and memory-consuming procedure in large-scale optimisation problems. Historically, the first response to this challenge was the development of variable metric \cite{davidon1991variable} and, especially, quasi-Newton \cite{dennis1977quasi} methods, which are preconditioned methods with $P_t$ being an approximation of $\nabla^2 f(x_t)$. Variable metric methods also include adaptive methods \cite{duchi2018introductory}, which are more studied in the stochastic (and online) optimisation setting and allow obtaining local guarantees \cite{lu2022adaptive}. However, these approaches only serve as inspiration and are beyond the scope of this paper. Quasi-Newton methods require the Hessian approximation to be asymptotically exact, which allows demonstrating that key properties of the Newton method remain valid under some conditions \cite{powell1976some}. If $f$ is strongly self-concordant, one can obtain explicit convergence rates of Broyden's family of quasi-Newton methods by taking into account the discrepancy between the Hessian and its approximation \cite{rodomanov2021greedy}. Some recent results \cite{kamzolov2023accelerated} are obtained by considering the inexact cubic regularized Newton method \cite{ghadimi2017second}. 
    At the same time, 
    variable metric methods, which allow the use of arbitrary (in general, stochastically biased) Hessian approximations, are understudied. Despite this, such approximations (for instance, Jacobi's preconditioning \cite{jacobi1845ueber} and Hutchinson's trace estimator \cite{hutchinson1989stochastic}) are used to design efficient optimisation methods \cite{jahani2021doubly,sadiev2022stochastic}. The latter methods inspire the preconditioning operation proposed in this paper with the aim to demonstrate the theoretical framework allowing a wider range of Hessian approximations and applicable to any variable metric method.

\section{Theory}

Let us consider optimisation problem \eqref{eq:problem_statement} for some Hilbert space $(E, \langle \cdot, \cdot\rangle)$ with induced norm $\|\cdot\|$ and dual space $E^*$ with dual norm $\|\cdot\|^*$, closed convex non-empty $Q \subseteq E$ and $f \in C^2(Q, \mathbb{R})$ with $M$-Lipschitz continuous w.r.t. $T$-Mahalanobis norm Hessian $\nabla^2 f \in \mathcal{L}(E, E^*)$, i.e. it holds that
\begin{equation} \label{eq:hess-cont}
    f(y) \leqslant f(x) + \langle \nabla f(x), y - x \rangle + \frac{1}{2} \|y - x\|_{\nabla^2 f(x)} + \frac{M}{6} \|y - x\|_{T(x)}^3,\quad \forall x, y \in Q
\end{equation}
for some $T: E \to \mathcal{L}(E, E^*)$ such that $T(x)$ is positive-definite $\forall x \in Q$. $\|x\|_{A}^2 := \langle A x, x \rangle$ for any $x \in E$ and $A \in \mathcal{L}(E, E^*)$. $\mathcal{L}(E, E^*)$ consists of bounded linear operators from $E$ to $E^*$. We assume that there exists second-order minimum $x^* \in Q$ of $f$, i.e. such that $\inf_{x \in Q} f(x) = f(x^*)$ and $\nabla^2 f(x^*)$ is positive-definite. $f$ is equipped with unbiased stochastic gradient oracle $g$ with bounded variance and stochastic Hessian approximation $H$, defined as follows.

Let $(\Xi, \mathcal{F}_\xi, P_\xi)$ and $(N, \mathcal{F}_\eta, P_\eta)$ be probability triples with independent $P_\xi$, $P_\eta$. $g: Q \times \Xi \to E^*$ is unbiased stochastic gradient oracle \cite{duchi2018introductory} (Definition 3.4.1), i.e. $\forall x \in Q$ it holds that 
\begin{equation} \label{eq:unbiasedness}
    \mathbb{E}_{P_\xi} [g(x, \xi)] = \nabla f(x).
\end{equation}
Hereinafter $\xi$ is a sample drawn from $P_\xi$. $g$ has bounded variance, i.e. $\forall x \in Q$ it holds that 
\begin{equation} \label{eq:grad-var}
    \mathbb{E}_{P_\xi}[\|g(x, \xi) - \nabla f(x)\|^*]^2 \leqslant \sigma_g^2 
\end{equation}
for some $\sigma_g > 0$. In turn, $H: Q \times N \to \mathcal{L}(E, E^*)$ is stochastic Hessian approximation, i.e. $\forall x \in Q$ it holds that
\begin{equation} \label{eq:hess-var}
    \mathbb{E}_{P_\eta}[\|H(x, \eta) - \nabla^2 f(x)\|_{op}^2] \leqslant \sigma_H^2
\end{equation}
for some $\sigma_H > 0$. Hereinafter $\eta$ is a sample drawn from $P_\eta$. $\|A\|_{op} := \sup_{\|x\| = 1} \|A x\|^*$ for any $A \in \mathcal{L}(E, E^*)$. Besides, we require $H(x, \eta)$ to be positive-definite for every $x, \eta$ in the entire domain of $H$.

Let us apply Linear Averaged Preconditioning (LAP) to Stochastic Gradient Descent (SGD). The proposed algorithm is listed in Algorithm~\ref{alg:sgd}. Note that if $E = \mathbb{R}^n$ for some $n \in \mathbb{N}$, one can choose $P_0$ to be equal to the $n \times n$ identity matrix. Also, note that the gradient step can be equivalently rewritten as follows:
\begin{equation}
    x_{t+1} \gets \operatorname{proj}_Q \left\{ x_t - \alpha_t P_t^{-1} g_t \right\}, 
\end{equation}
where $\operatorname{proj}_Q$ is a projection onto $Q$, which is well-defined due to the Hilbert projection theorem. It is worth noticing that the averaging in Algorithm~\ref{alg:sgd} is an exponential moving average with a variable smoothing constant $\beta_t$. This is the only averaging to consider if greedy-optimal preconditioning is to be found. Indeed, the decrease of the objective function is upper bounded by the Taylor expansion, which is dependent only on $P_t$ and $\nabla^2 f(x_t)$, so no previous preconditionings have an impact on this upper bound. Further, there is no $P_{t+1}^\prime$ that ensures a better decrease of the objective function than any $P_t + \beta (H_t - P_t)$, $\beta \in [0, 1]$ (preconditioning operator lying on a segment between $P_t$ and $H_t$ in $\mathcal{L}(E, E^*)$) because of a geometrical reason: there exists $P_{t+1}$ on this segment, such that 
$$\|P_{t+1} - P_t\|_{op} \leqslant \|P_{t+1}^\prime - P_t\|_{op} \ \mbox{and}  \ \|P_{t+1} - H_t\|_{op} \leqslant \|P_{t+1}^\prime - H_t\|_{op}$$
with at least one inequality being strict (the orthogonal projection of $P_{t+1}^\prime$ onto the segment satisfies this due to the Hilbert projection theorem). At the same time, the decrease of the objective function depends only on the norms of operators $P_{t+1} - P_t$ and $P_{t+1} - H_t$ along certain directions, which both are upper bounded by their operator norms $\|\cdot\|_{op}$. Note that $P_t$ is positive-definite for all $t = 0, ..., T$ by construction.

\begin{algorithm}
    \caption{SGD with LAP} \label{alg:sgd}
    \begin{algorithmic}
    \Require initial point $x_0 \in Q$
    \State initial positive-definite preconditioning $P_{-1} \in \mathcal{L}(E, E^*)$
    \State iterations number $T > 0$
    \State step sizes $\{\alpha_t > 0 : t=0,...,T-1\}$
    \State averaging rates $\{\beta_t \in [0, 1] : t=0,...,T-1\}$
    \Ensure draw sample $t \in \{0, ..., T-1\}$ from $P_T$, defined by $P_T(t) := \alpha_t / \sum_{k=0}^{T-1} \alpha_k$, $x_t$
    \For{$t = 0, ..., T-1$}
        \State draw samples $\xi_t$, $\eta_t$ from $P_\xi$, $P_\eta$, correspondingly
        \State evaluate oracles $g_t \gets g(x_t, \xi_t)$, $H_t \gets H(x_t, \eta_t)$
        \State average preconditioning $P_t \gets P_{t-1} + \beta_t (H_t - P_{t-1})$
        \State make gradient step $x_{t+1} \gets \arg \min_{x \in Q} \left\{\alpha_t \langle g_t, x - x_t \rangle + \frac{1}{2} \|x - x_t\|^2_{P_t}\right\}$
    \EndFor
    \end{algorithmic}
\end{algorithm}

Firstly, let us consider problems of optimising non-convex functions given the biased, i.e. not satisfying condition \eqref{eq:unbiasedness}, stochastic gradient oracle. Plain stochastic gradient descent applied to functions with $L$-Lipschitz continuous gradient requires $\alpha_t = O(1 / L)$ and ensures $O(L (f(x_0) - f(x^*)) / T) + O(\sigma_g^2)$ convergence rate for squared norm of the gradient, where $g \leqslant O(f)$ means that there exists $c \in \mathbb{R}_+$ such that $g(t) \leqslant c f(t)$ for all $t \in \mathbb{N}$. Intuitively, preconditioning must let choose $\alpha_t$ closer to $1$. Theorem~\ref{th:sgd} below gives explicit formula for $\alpha_t$, such that $1/\alpha_t - 1$ tends to value, proportional to $\sigma_H$, $\sigma_T$ or $\sigma_g$, conditions allowing arbitrary choice of $\beta_t \in (0, 1)$, and non-asymptotic convergence rate of Algorithm~\ref{alg:sgd} for expectation of $(\|\nabla f(x_t)\|_{P_t}^*)^2$.

\begin{theorem}[Convergence rate of SGD with LAP, non-convex function] \label{th:sgd}
If $\beta_0 = 1$, $\beta_t \in [0, 1]$, and $\lambda_{\min} (P_t) \geqslant \mu_t$ for all $t = 1, ..., T-1$, where sequence $\{\mu_t > 0 : t = 1, ..., T - 1\}$ is chosen satisfying the following condition:
\begin{equation*}
    \frac{1}{\mu_t} > \frac{1}{\mu_{t-1}} + \frac{M \alpha_{t-1} \|g_{t-1}\|_{P_{t-1}}^* (1 + \sigma_T^{1/2})}{\Delta_{t-1}},
\end{equation*}
where $\mu_0 > 0$, $\Delta_0 = \sigma_H / \mu_0$, and sequence $\{\Delta_t > 0 : t = 0, ..., T - 1\}$ is decreasing, and where step sizes are chosen satisfying the following condition:
\begin{equation*}
    \alpha_t \leqslant \frac{1}{1 + \left(1 + \sqrt{M \|g_t\|_{P_t}^*}\right) \Delta_t + \sqrt{M \|g_t\|_{P_t}^*} (\sigma_T + 2)},
\end{equation*}
it holds that
\begin{align*}
    \mathbb{E}_{P_T P_\xi P_\eta}\left[\|\nabla f(x_t)\|_{P_t}^*\right]^2 &\leqslant \frac{2 (f(x_0) - f(x^*))}{T} \left(1 + \frac{\sigma_H}{\mu_0} + \sqrt{M B} \left(\sigma_T + \frac{\sigma_H}{\mu_0} + 2\right)\right) + \frac{3 \sigma_g^2}{\mathbb{E}_{P_T} [\mu_{t^\prime}]},
\end{align*}
where $B := \max_{k=0,...,T-1} \|g_k\|^*_{P_k} > 0$.
\end{theorem}

Secondly, let us consider convex optimisation problems with unbiased stochastic gradient oracle. Convergence rate of plain stochastic gradient method is guaranteed to be $O(R^2 L / T) + O(\sigma R / \sqrt{T})$ in this case \cite{lan2012optimal}, which is achieved by choice $\alpha_t = O(1 / \sqrt{T})$, in addition to $O(1 / L)$. Theorem~\ref{th:sgd-cvx} gives conditions on $\beta_t$ and $\alpha_t$ ensuring similar convergence rate for SGD with LAP, which are more strict than in non-convex case, and provide non-asymptotic convergence rate of Algorithm~\ref{alg:sgd} for expectation of $f(x_t) - f(x^*)$.

\begin{theorem}[Convergence rate of SGD with LAP, convex function] \label{th:sgd-cvx}
If $\beta_0 = 1$, $\beta_t \in [0, 1]$, and $\lambda_{\min} (P_t) \geqslant \mu > 0$ for all $t = 1, ..., T-1$, and sequence $\{\Delta_t > 0 : t = 0, ..., T - 1\}$ is decreasing, and where step sizes and averaging rates are chosen satisfying the following condition:
\begin{equation*}
    \alpha_t = \min\left\{\frac{1}{\frac{1}{c_1}(1 + \Delta_t) + (\|g_t\|_{P_t}^*)^{3/4} (1 + (\Delta_t + \sigma_T)^{3/4})}, \sqrt{\frac{2 c_3 \sigma_H + c_3 M (1 + \sigma_T^{1/2}) + c_2 + 1}{T \left(\frac{\sigma_g^2}{\mu (1 - c_1)} + \frac{M}{2}\right)}}\right\},
\end{equation*}
\begin{equation*}
    \frac{1}{1 - \beta_t} \geqslant \max\left\{1 + \frac{\alpha_{t-1} M \|g_{t-1}\|_{P_{t-1}}^* (1 + \sigma_T^{1/2}) + 2 \sigma_H}{\|H_{t-1} - P_{t-1}\|_{op}}, (t+1)^2\right\},
\end{equation*}
where $c_1$, $c_2$, and $c_3$ denote some small positive constants, it holds that
\begin{align*}
    &\mathbb{E}_{P_\eta P_\xi P_T} [f(x_t) - f(x^*)] \leqslant O(1) \sqrt{\frac{(\sigma_g^2/\mu + M B^{3/2}) (\sigma_H + M (1 + \sigma_T^{1/2}) + 1)}{T}}\\
    &\quad\quad+ O(1) \left(1 + \sigma_H + (1 + (\sigma_H + \sigma_T)^{3/4}) B^{3/4}\right) \frac{R^2 (\sigma_H + M (1 + \sigma_T^{1/2}) + 1)}{T},
\end{align*}
where $R := \|x_0 - x^*\|_{H_0} > 0$ and $O(1)$ denotes some small positive constant, hiding $c_1$, $c_2$, and $c_3$.
\end{theorem}

Summarising Theorems~\ref{th:sgd} and \ref{th:sgd-cvx}, it worth noting that Lipschitz constant of the gradient $L$, included in convergence rates of plain SGD, was replaced everywhere by the norm of the gradient $B$, which is decreasing while approaching to $x^*$, together with relative Lipschitz constant of the Hessian $M$. This allows to address problems of optimising function which gradient is not Lipschitz-continuous, i.e. growth of gradient is not limited. Besides, dependency of convergence rates on $B$ improves upon that of $L$ for plain SGD. Omitting terms proportional to $\sigma_H$, $\sigma_T$ or $\sigma_g$, for non-convex functions, decreasing term of convergence rate is $O((1 + \sqrt{M B}) (f(x_0) - f(x^*)) / T)$, and for convex functions, convergence rate is $O(M B^{3/4} (1 / \sqrt{T} + R^2 / T))$.

\section{Practical example}

Let us consider the following problem:
\begin{equation} \label{eq:test-quad}
    \min_{x \in \mathbb{R}^n} \langle A x, x \rangle - \langle b, x\rangle, 
\end{equation}
where $A \in \mathbb{R}_+^{n \times n}$ is positive definite with uniform random non-diagonal elements drawn from $(0, 0.001)$ and uniform random diagonal elements drawn from $(0, 0.006)$, and $b \in \mathbb{R}_+^n$ with uniform random elements drawn from $(0, 1)$. 

Figure~\ref{fig:quad} shows convergence curves of SGD and SGD with LAP, Algorithm~\ref{alg:sgd}, with vertical axis measuring function value $f(x_t)$ (left, $f(x^*) = 0$) or norm of the gradient $\|\nabla^2 f(x_t)\|_2$ (right) in logarithmic scale, and horizontal axis measuring iterations number $t$. One can see that preconditioning with properly chosen $\beta_t$, $t=0, ..., T-1$ improves upon plain stochastic gradient descent, even if $\sigma_H$ is significantly big (bigger than $\lambda_{\max}(A)$, in our experiment). Note, that the case of biased stochastic gradient oracle is considered and therefore both algorithms do not converge to $x^*$, but only to its vicinity.

\begin{figure}[H]
     \centering
     \begin{subfigure}[b]{0.45\textwidth}
         \centering
         \includegraphics[width=\textwidth]{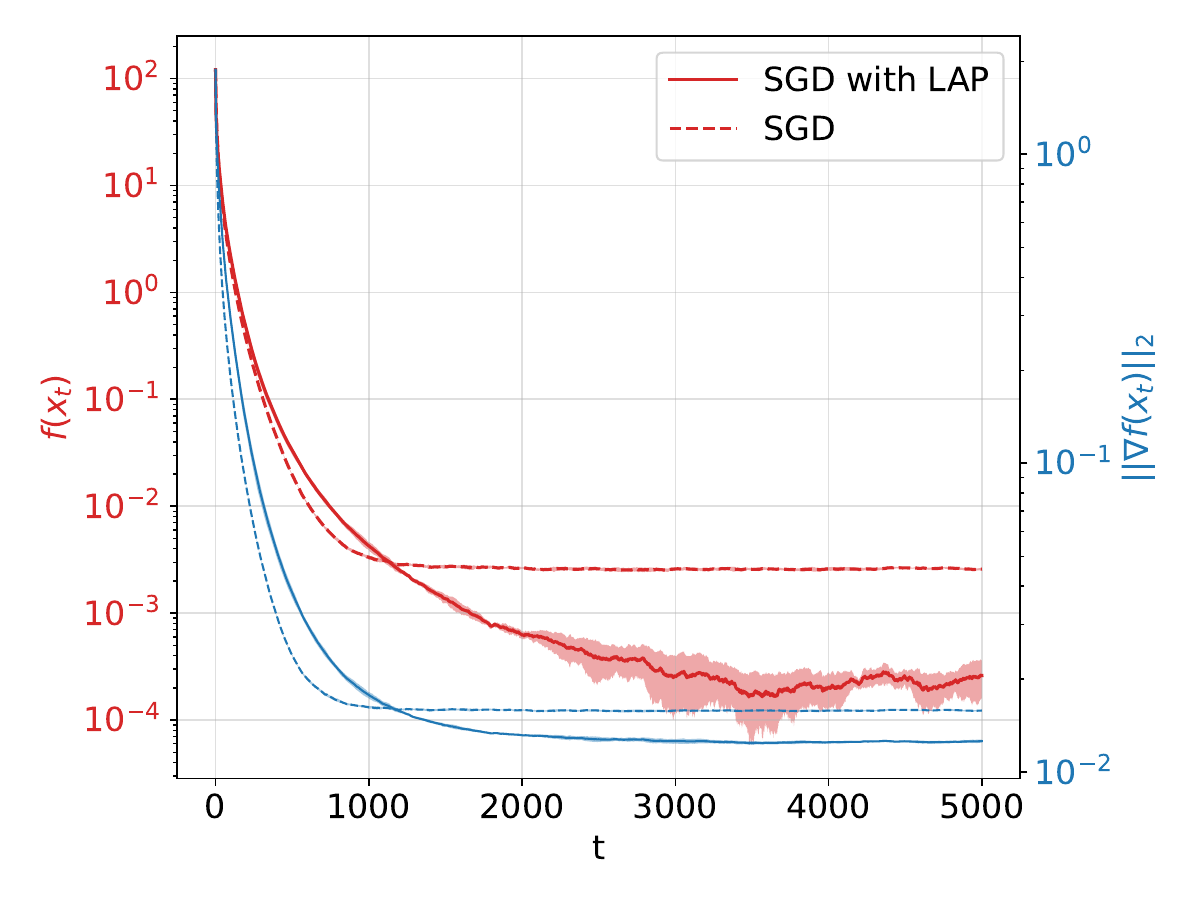}
         \caption{$\beta_t \in (0, 1), t = 0, ..., T-1$, $\alpha_0 = 0.8$}
     \end{subfigure}
     \begin{subfigure}[b]{0.45\textwidth}
         \centering
         \includegraphics[width=\textwidth]{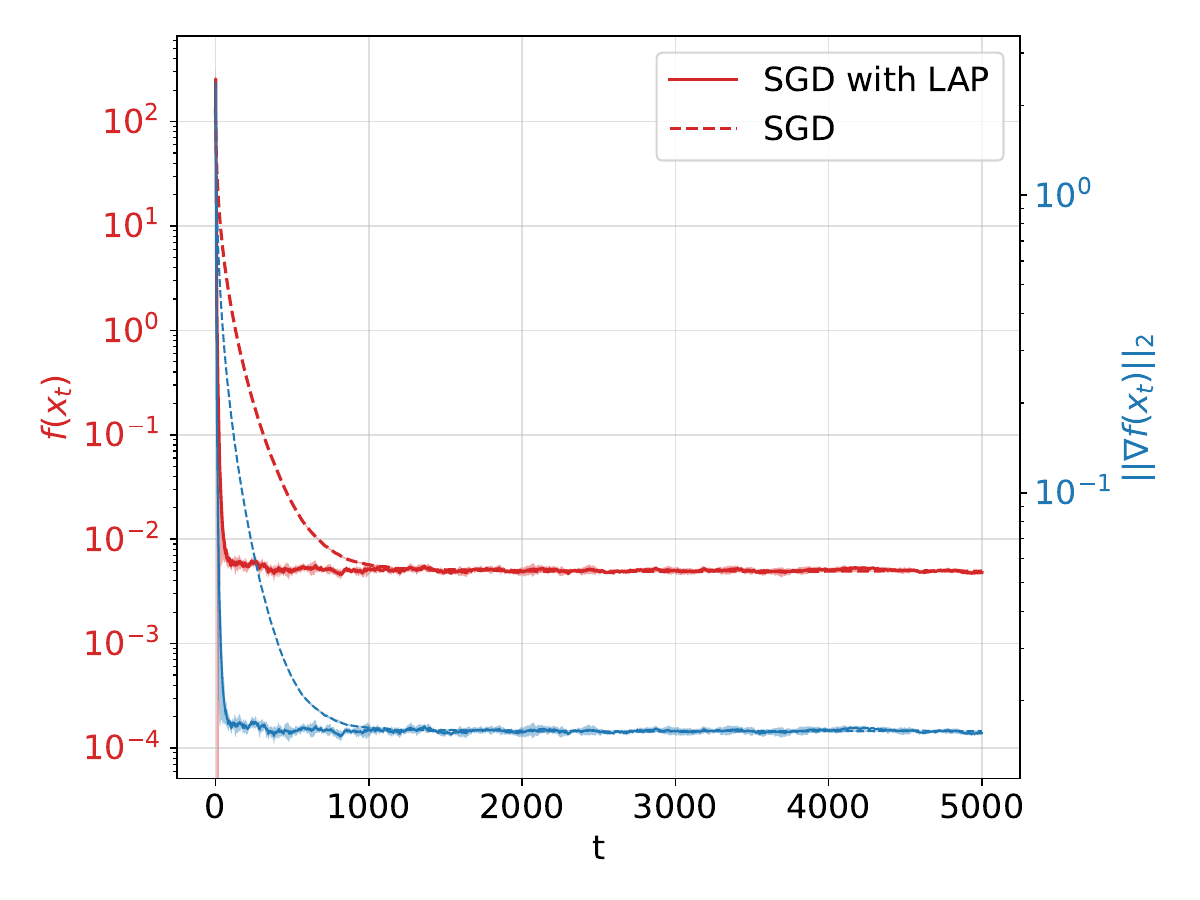}
         \caption{$\beta_t \in \left(1 - \frac{1}{t^2}, 1\right), t = 0, ..., T-1$, $\alpha_0 = 4$}
     \end{subfigure}
        \caption{Convergence curves of SGD with and without LAP applied to \eqref{eq:test-quad} with $T = 5000$, $n = 10$, $\sigma_g = 0.004$ (biased), $\sigma_H = 1$, $\mu_t > 0.008$, $\beta_t \approx \arg \min_{\beta} f(x_{t+1})$ using Brent's method on a corresponding interval, and $\alpha_t = \alpha_0 / \sqrt{t}$ with $\alpha_0 = 10$ for SGD and corresponding $\alpha_0$ for SGD with LAP. $\alpha_0$ are chosen to be optimal using grid search.}
        \label{fig:quad}
\end{figure}

\section{Discussion}

This paper continues the development of theoretical framework for analysing preconditioning in stochastic gradient methods, initiated by authors in \cite{pasechnyuk2022effects}. 

Theorem~\ref{th:sgd} improves upon previous analysis of stochastic gradient descent with preconditioning in terms of:
\begin{enumerate}
    \item requirements on $\beta_t$, $t = 1, ..., T-1$, which were completely compensated by requirements on margins from zero $\mu_t$, $t = 1, ..., T-1$;
    \item conditions on Hessian and its approximation, which are now presented in ``additive'' notation through $\|\nabla^2 f(x_t) - P_t\|_{op, P_t}$ instead of previously used ``multiplicative'' notation through bounding $\|\nabla^2 f(x_t) P_t^{-1}\|_{op}$, which narrowed the class of acceptable $\nabla^2 f(x_t)$.
\end{enumerate}

Newly initiated direction of work is analysis of stochastic gradient descent with preconditioning for convex functions in Theorem~\ref{th:sgd-cvx}. The proposed analysis improves upon \cite{scheinberg2016practical} by ensuring the absence of diverging terms. The dependency on $B$ and requirements on $\beta_t$, $t = 1, ..., T-1$ in the Theorem~\ref{th:sgd-cvx} are to be improved in future works.

\bibliography{main}

\appendix
\section{Proofs}

\begin{lemma}[Non-expansiveness of $\operatorname{proj}_Q$ \cite{bertsekas2009convex}]
    For any closed convex non-empty $Q \subseteq E$, $x \in Q$, and $y \in E$, it holds that
    \begin{equation} \label{eq:proj-exp}
        \|\operatorname{proj}_Q y - x\| \leqslant \|y - x\|.
    \end{equation}
\end{lemma}

\begin{lemma}[Direction of $\operatorname{proj}_Q$]
    For any closed convex non-empty $Q \subseteq E$, $x \in Q$, $g \in E^*$, and positive-definite $P \in \mathcal{L}(E, E^*)$, it holds that
    \begin{equation} \label{eq:proj-dir}
        \langle g, \operatorname{proj}_Q (x - P^{-1} g) - (x - P^{-1} g) \rangle \geqslant 0.
    \end{equation}
\end{lemma}
\begin{proof}
    It follows from that, in the same context and any $h \in E$, it holds that
    \begin{equation*}
        \langle h, \operatorname{proj}_Q (x - h) - (x - h) \rangle \geqslant 0,
    \end{equation*}
    by applying it to $h = P^{-1} g$, due to the equivalence of $\|\cdot\|$ and $\|\cdot\|_P^*$ together with self-adjointness of $P^{-1}$ (implied by its positive-definiteness).
\end{proof}

\begin{lemma}[Length of $\operatorname{proj}_Q$]
    For any closed convex non-empty $Q \subseteq E$, $x \in Q$, $g, d \in E^*$, and positive-definite $P \in \mathcal{L}(E, E^*)$, it holds that
    \begin{equation} \label{eq:proj-len}
        |\langle d, (x - P^{-1} g) - \operatorname{proj}_Q (x - P^{-1} g) \rangle| \leqslant \|d\|_P^* \|g\|_P^*.
    \end{equation}
\end{lemma}
\begin{proof}
    It follows from extended Cauchy--Schwarz inequality.
\end{proof}

\begin{lemma}[Correction to coefficient of cubic Taylor's expansion term]
    For any $g \in E^*$, and positive-definite $P, T \in \mathcal{L}(E, E^*)$, it holds that
    \begin{equation} \label{eq:upb-cubic}
        \|P^{-1} g\|^3_T \leqslant \frac{3}{2} (\|g\|_{P}^*)^3 (1 + |\|T - P\|_{op, P}|^{3/2}).
    \end{equation}
\end{lemma}
\begin{proof}
    Firstly,
    \begin{align*}
        \langle P^{-1} g, T P^{-1} g \rangle &= \langle P^{-1} g, g \rangle + \langle P^{-1} g, (T - P) P^{-1} g \rangle\\
        &\leqslant (\|g\|_P^*)^2 + \|T - P\|_{op, P} (\|g\|_P^*)^2\\
        &\leqslant (\|g\|_P^*)^2 (1 + |\|T - P\|_{op, P}|)
    \end{align*}
    The fact that $(1 + x)^{3/2} \leqslant \frac{3}{2} (1 + x^{3/2})$ for any $x \in \mathbb{R}_+$ implies lemma's statement.
\end{proof}

\begin{lemma}[Harmonic mean upper bound]
    For any $I \in \mathbb{N}$ and sequence $\{a_i \in \mathbb{R}_+ : i = 1, ..., I\}$, it holds that
    \begin{equation} \label{eq:upb-harmonic}
        \frac{I}{\sum_{i=1}^I \frac{1}{a_i}} \leqslant \frac{\sum_{i=1}^I a_i}{I}.
    \end{equation}
\end{lemma}

\begin{lemma}[Abel's inequality]
    For any $I \in \mathbb{N}$ and sequences $\{a_i \in \mathbb{R} : i = 1, ..., I\}$, $\{b_i \in \mathbb{R} : i = 1, ..., I\}$, it holds that
    \begin{gather} \label{eq:abel}
        \sum_{i=1}^I a_i b_i = \left(\sum_{i=1}^I b_i\right) a_I - \sum_{j=1}^{I-1} \left(\sum_{i=1}^j b_i\right) (a_{j+1} - a_j),\\
        \nonumber\sum_{i=1}^I a_i b_i = \left(\sum_{i=1}^I b_i\right) a_0 - \sum_{j=1}^{I-1} \left(\sum_{i=1}^j b_i\right) (a_j - a_{j+1}).
    \end{gather}
\end{lemma}

\begin{lemma}[Abel's test]
    For any $I \in \mathbb{N}$ and sequences $\{a_i \in \mathbb{R} : i = 1, ..., +\infty\}$, $\{b_i \in \mathbb{R}_+ : i = 1, ..., +\infty\}$, such that $b_{i+1} \leqslant b_i$ for any $i = 1, ..., +\infty$ and series $\sum_{i=1}^{+\infty} a_i$ converges, it holds that there exists $c \in \mathbb{R}_+$, such that
    \begin{equation} \label{eq:abel-test}
        \sum_{i=1}^{\infty} a_i b_i = c.
    \end{equation}
\end{lemma}

\begin{prop}
    For any $a, b \in E^*$ and positive-definite $P \in \mathcal{L}(E, E^*)$, it holds that
    \begin{equation} \label{eq:brackets}
        -\langle a, P^{-1} b \rangle = -\frac{1}{2} \|a\|_P^* - \frac{1}{2} \|b\|_P^* + \frac{1}{2} \|a - b\|_P^*.
    \end{equation}
\end{prop}

\begin{prop}
    For any $a, x \in \mathbb{R}_+$, it holds that \inlineequation[eq:upb-var]{a x + \frac{a^2}{2} \leqslant \frac{x^2}{4} + \frac{3 a^2}{2}}.
\end{prop}

\begin{prop}
    For any $A, B \in \mathcal{L}(E, E^*)$ and positive-definite $P \in \mathcal{L}(E, E^*)$, such that $\mu \leqslant \lambda_{\min}(P)$, it holds that
    \begin{gather} \label{eq:op-b}
        \mu \frac{\langle A x, x \rangle}{\langle P x, x \rangle} \leqslant \frac{\langle A x, x \rangle}{\langle x, x \rangle} \leqslant \|A\|_{op},\\
        \nonumber \langle B x, x\rangle \leqslant \|B\|_{op} \langle x, x \rangle \leqslant \frac{\|B\|_{op}}{\mu}\langle P x, x\rangle.
    \end{gather}
\end{prop}

\begin{proof}[Proof of Theorem~\ref{th:sgd} via diminishing step sizes]

Let us introduce $\Delta_t > 0$, such that
\begin{equation} \label{eq:delta}
    \|\nabla^2 f(x_t) - P_t\|_{op, P_t} \leqslant \Delta_t.
\end{equation}
$\|A\|_{op, P} := \sup_{\|x\|_P = 1} \|A x\|$ for any positive-definite $P \in \mathcal{L}(E, E^*)$ and $A \in \mathcal{L}(E, E^*)$ (not necessarily positive-definite). Let us also introduce auxiliary point
\begin{equation} \label{eq:aux-point}
    x_{t+1}^\prime := x_t - P_t^{-1} g_t.
\end{equation}
\begin{align}
    \nonumber f(x_{t+1}) &\leqslant f(x_t) + \langle \nabla f(x_t), x_{t+1} - x_t \rangle + \frac{1}{2} \|x_{t+1} - x_t\|_{\nabla^2 f(x_t)}^2 + \frac{M}{6} \|x_{t+1} - x_t\|_{T(x_t)}\\
    \nonumber&\stackrel{\eqref{eq:proj-exp}}{\leqslant} f(x_t) + \langle \nabla f(x_t), x_{t+1} - x_t \rangle + \frac{\alpha_t^2}{2} \|P_t^{-1} g_t\|_{\nabla^2 f(x_t)}^2 + \frac{\alpha_t^3 M}{6} \|P_t^{-1} g_t\|_{T(x_t)}\\
    \nonumber&= f(x_t) + \langle \nabla f(x_t), x_{t+1} - x_t \rangle + \frac{\alpha_t^2}{2} (\|g_t\|_{P_t}^*)^2 +\\
    \nonumber&\quad + \frac{\alpha_t^2}{2}\langle (\nabla^2 f(x_t) - P_t) P_t^{-1} g_t, P_t^{-1} g_t \rangle + \frac{\alpha_t^3 M}{6} \|P_t^{-1} g_t\|_{T(x_t)}\\
    \nonumber&\stackrel{\eqref{eq:delta}}{\leqslant} f(x_t) + \langle \nabla f(x_t), x_{t+1} - x_t \rangle + \frac{\alpha_t^2 (1 + \Delta_t)}{2} (\|g_t\|_{P_t}^*)^2 + \frac{\alpha_t^3 M}{6} \|P_t^{-1} g_t\|_{T(x_t)}^3\\
    \nonumber&= f(x_t) + \langle \nabla f(x_t), x_{t+1}^\prime - x_t\rangle + \frac{\alpha_t^2 (1 + \Delta_t)}{2} (\|g_t\|_{P_t}^*)^2 + \frac{\alpha_t^3 M}{6} \|P_t^{-1} g_t\|_{T(x_t)}^3 +\\
    \nonumber&\quad+ \langle \nabla f(x_t) - g_t, x_{t+1} - x_{t+1}^\prime \rangle - \langle g_t, x_{t+1}^\prime - x_{t+1}\rangle\\
    \nonumber&\stackrel{\eqref{eq:proj-dir},\eqref{eq:aux-point}}{\leqslant} f(x_t) - \alpha_t \langle \nabla f(x_t), P_t^{-1} g_t \rangle + \frac{\alpha_t^2 (1 + \Delta_t)}{2} (\|g_t\|_{P_t}^*)^2 + \frac{\alpha_t^3 M}{6} \|P_t^{-1} g_t\|_{T(x_t)}^3 +\\
    \nonumber&\quad+ \langle \nabla f(x_t) - g_t, x_{t+1} - x_{t+1}^\prime \rangle\\
    \nonumber&\stackrel{\eqref{eq:brackets}}{=} f(x_t) - \frac{\alpha_t}{2} (\|\nabla f(x_t)\|_{P_t}^*)^2 + \left(\frac{\alpha_t^2 (1 + \Delta_t)}{2} - \frac{\alpha_t}{2}\right) (\|g_t\|_{P_t}^*)^2 + \frac{\alpha_t^3 M}{6} \|P_t^{-1} g_t\|_{T(x_t)}^3 +\\
    \nonumber&\quad+ \langle \nabla f(x_t) - g_t, x_{t+1} - x_{t+1}^\prime \rangle + \frac{\alpha_t}{2} (\|\nabla f(x_t) - g_t\|_{P_t}^*)^2\\
    \nonumber&\stackrel{\eqref{eq:proj-len}}{\leqslant} f(x_t) - \frac{\alpha_t}{2} (\|\nabla f(x_t)\|_{P_t}^*)^2 + \left(\frac{\alpha_t^2 (1 + \Delta_t)}{2} - \frac{\alpha_t}{2}\right) (\|g_t\|_{P_t}^*)^2 + \frac{\alpha_t^3 M}{6} \|P_t^{-1} g_t\|_{T(x_t)}^3 +\\
    \nonumber&\quad+ \alpha_t \|\nabla f(x_t) - g_t\|_{P_t}^* \|g_t\|_{P_t}^* + \frac{\alpha_t}{2} (\|\nabla f(x_t) - g_t\|_{P_t}^*)^2\\
    \nonumber&\stackrel{\eqref{eq:upb-var}}{\leqslant} f(x_t) - \frac{\alpha_t}{2} (\|\nabla f(x_t)\|_{P_t}^*)^2 + \left(\frac{\alpha_t^2 (1 + \Delta_t)}{2} - \frac{\alpha_t}{4}\right) (\|g_t\|_{P_t}^*)^2 + \frac{\alpha_t^3 M}{6} \|P_t^{-1} g_t\|_{T(x_t)}^3 +\\
    \nonumber&\quad+ \frac{3\alpha_t}{2} (\|\nabla f(x_t) - g_t\|_{P_t}^*)^2\\
    \label{eq:descent}&\stackrel{\eqref{eq:upb-cubic}}{\leqslant} f(x_t) - \frac{\alpha_t}{2} (\|\nabla f(x_t)\|_{P_t}^*)^2 + \frac{\alpha_t}{4} Q_t(\alpha_t) (\|g_t\|_{P_t}^*)^2 + \frac{3\alpha_t}{2} (\|\nabla f(x_t) - g_t\|_{P_t}^*)^2,
\end{align}
where $Q_t$ is defined as follows:
\begin{equation} \label{eq:quadr}
    Q_t(\alpha) := M \|g_t\|_{P_t}^* (1 + (\Delta_t + \sigma_T)^{3/2}) \cdot \alpha^2 + 2 (1 + \Delta_t) \cdot \alpha - 1.
\end{equation}
We choose $\alpha_t$ small enough to make $Q_t(\alpha_t) \leqslant 0$. Since coefficients of $Q_t(\alpha)$ of $\alpha^2$ and $\alpha$ are positive, it is sufficient to choose $\alpha_t$ less than the larger of two its roots. Using the fact that $\sqrt{x + y} \leqslant \sqrt{x} + \sqrt{y}$ for any $x, y \in \mathbb{R}_+$, we find that it sufficient to choose $\alpha_t$ such that 
\begin{equation*}
    \alpha_t \leqslant \frac{1}{1 + \Delta_t + (M \|g_t\|_{P_t}^*)^{1/2} (1 + (\Delta_t + \sigma_T)^{3/4})}.
\end{equation*}
To make dependency of $\alpha_t$ on $\Delta_t$ simpler, we apply the fact that $x^{3/4} \leqslant \frac{3}{4}x + \frac{1}{4}$ for any $x \in \mathbb{R}_+$, to find that it sufficient to choose $\alpha_t$ such that 
\begin{equation} \label{eq:alpha_sgd}
    \alpha_t \leqslant \frac{1}{1 + \left(1 + \frac{3}{4}(M \|g_t\|_{P_t}^*)^{1/2}\right) \Delta_t + (M \|g_t\|_{P_t}^*)^{1/2} \left(\frac{3}{4}\sigma_T + \frac{5}{4}\right)}.
\end{equation}

\eqref{eq:descent} and \eqref{eq:alpha_sgd} imply the following ``descent lemma'' inequality:
\begin{equation} \label{eq:descent-lemma}
    f(x_{t+1}) \leqslant f(x_t) - \frac{\alpha_t}{2} (\|\nabla f(x_t)\|_{P_t}^*)^2 + \frac{3 \alpha_t}{2} (\|\nabla f(x_t) - g_t\|_{P_t}^*)^2.
\end{equation}
Rearranging, summing up for $t = 0, ..., T-1$, and taking into account that $f(x_{t+1}) \geqslant f(x^*)$,
\begin{align*}
    \sum_{t=0}^{T-1} \alpha_t (\|\nabla f(x_t)\|_{P_t}^*)^2 &\leqslant 2 (f(x_0) - f(x^*)) + 3 \sum_{t=0}^{T-1} \alpha_t (\|\nabla f(x_t) - g_t\|_{P_t}^*)^2.
\end{align*}
Taking $\mathbb{E}_{P_\xi}$ and $\mathbb{E}_{P_T}$, and dividing by $\sum_{k=0}^{T-1} \alpha_t$,
\begin{equation} \label{eq:sgd-implicit}
    \mathbb{E}_{P_T P_\xi}[\|\nabla f(x_t)\|_{P_t}^*]^2 \leqslant \frac{2 (f(x_0) - f(x^*))}{\sum_{k=0}^{T-1} \alpha_k} + 3 \mathbb{E}_{P_T P_\xi}[\|\nabla f(x_{t^\prime}) - v_{t^\prime}\|_{P_{t^\prime}}^*]^2.
\end{equation}

Let us introduce the following denotation:
\begin{equation} \label{eq:integrant}
    \Sigma_t := \sum_{k=0}^{t-1} (\|v_k\|_{P_k}^*)^{1/2}.
\end{equation}
\begin{align}
    \nonumber\frac{2 T}{\sum_{k=0}^{T-1} \alpha_k} &\stackrel{\eqref{eq:upb-harmonic}}{\leqslant} \frac{2 \sum_{k=0}^{T-1} \frac{1}{\alpha_k}}{T} \stackrel{\eqref{eq:alpha_sgd}}{=} 2 + \frac{2 \sum_{k=0}^{T-1} (1 + \frac{3}{4} (M \|v_k\|_{P_k}^*)^{1/2}) \Delta_t}{T} + M^{1/2} \left(\frac{3 \sigma_T}{2} + \frac{5}{2}\right)\frac{\Sigma_T}{T}\\
    \label{eq:means}&=2 + \frac{2 \sum_{k=0}^{T-1} \Delta_k}{T} + M^{1/2} \frac{3}{2} \frac{\sum_{k=0}^{T-1} (\|v_k\|_{P_k}^*)^{1/2} \Delta_t}{T} + M^{1/2} \left(\frac{3 \sigma_T}{2} + \frac{5}{2}\right)\frac{\Sigma_T}{T},
\end{align}
\begin{align*}
    \mathbb{E}_{P_\eta} &[\|\nabla^2 f(x_t) - P_t\|_{op}] = \mathbb{E}_{P_\eta} [\|\nabla^2 f(x_t) - H_t + (1 - \beta_t) (H_t - P_{t-1})\|_{op}]\\
    &\qquad\qquad\qquad= \mathbb{E}_{P_\eta} [\|\beta_t (\nabla^2 f(x_t) - H_t) + (1 - \beta_t) (\nabla^2 f(x_t) - P_{t-1})\|_{op}]\\
    &\qquad\qquad\qquad\stackrel{\eqref{eq:hess-var}, \eqref{eq:hess-cont}}{\leqslant} \beta_t \sigma_H + (1 - \beta_t) M\|x_t - x_{t-1}\|_{T(x_{t-1})} + (1 - \beta_t) \|\nabla^2 f(x_{t-1}) - P_{t-1}\|_{op}\\
    &\qquad\qquad\qquad\stackrel{\eqref{eq:proj-exp}}{\leqslant} \beta_t \sigma_H + (1 - \beta_t) M \alpha_{t-1} \|P_{t-1}^{-1} g_{y-1}\|_{T(x_{t-1})} + (1 - \beta_t) \|\nabla^2 f(x_{t-1}) - P_{t-1}\|_{op}\\
    &\qquad\qquad\qquad\stackrel{\eqref{eq:upb-cubic}}{\leqslant} \beta_t \sigma_H + (1 - \beta_t) M \alpha_{t-1} \|g_{t-1}\|^*_{P_{t-1}} (1 + \sigma_T^{1/2}) + (1 - \beta_t) \|\nabla^2 f(x_{t-1}) - P_{t-1}\|_{op}.
\end{align*}
\begin{align*}
    \mathbb{E}_{P_\eta} [\|\nabla^2 f(x_t) - P_t\|_{op, P_t}] &\stackrel{\eqref{eq:op-b}}{\leqslant} \mu_t \beta_t \sigma_H + \mu_t (1 - \beta_t) M \alpha_{t-1} \|g_{t-1}\|_{P_{t-1}}^* (1 + \sigma_T^{1/2})\\
    &\quad+ \frac{\mu_t}{\mu_{t-1}} (1 - \beta_t) \|\nabla^2 f(x_{t-1}) - P_{t-1}\|_{op, P_{t-1}}.
\end{align*}
\eqref{eq:means}, together with \eqref{eq:integrant}, implies
\begin{equation} \label{eq:means-pu-1}
    \frac{2 T}{\sum_{k=0}^{T-1} \alpha_k} \leqslant 2 + \frac{2 \sum_{k=0}^{T-1} \Delta_k}{T} - M^{1/2} \frac{3}{2}\frac{\sum_{k=1}^{T-1} \Sigma_k (\Delta_{k-1} - \Delta_k)}{T} + M^{1/2} \left(\frac{3 (\sigma_T + \Delta_0)}{2} + \frac{5}{2}\right)\frac{\Sigma_T}{T}.
\end{equation}
Let us ensure that sequence $\Delta_t$ is monotone decreasing. It holds if $\beta_t$ satisfies the following condition:
\begin{align*}
    \beta_t \geqslant 1 - \frac{\Delta_{t-1}/\mu_t - \sigma_H}{\Delta_{t-1}/\mu_{t-1} + M \alpha_{t-1} \|g_{t-1}\|_{P_{t-1}}^* (1 + \sigma_T^{1/2}) - \sigma_H}.
\end{align*}
Note, that this condition reduces to $\beta_t \geqslant 0$, if
\begin{align*}
    \frac{1}{\mu_t} > \frac{1}{\mu_{t-1}} + \frac{M \alpha_{t-1} \|g_{t-1}\|_{P_{t-1}}^* (1 + \sigma_T^{1/2})}{\Delta_{t-1}}.
\end{align*}
Taking $\mathbb{E}_{P_\eta}$ in \eqref{eq:means-pu-1}, and setting $\beta_0 = 1$,
\begin{align*}
    \frac{2 T}{\sum_{k=0}^{T-1} \alpha_k} &\leqslant 2 + \frac{2\sum_{k=0}^{T-1} \Delta_k}{T} + M^{1/2} \frac{3 (\sigma_T + \Delta_0) + 5}{2} \frac{\Sigma_T}{T}\\
    &\leqslant 2 (1 + \Delta_0) + (M B)^{1/2} \frac{3 (\sigma_T + \Delta_0) + 5}{2}\\
    &\stackrel{\eqref{eq:op-b}}{\leqslant} 2 \left(1 + \frac{\sigma_H}{\mu_0} + (M B)^{1/2} \left(\sigma_T + \frac{\sigma_H}{\mu_0} + 2\right)\right).
\end{align*}
\end{proof}

\begin{proof}[Proof of Theorem~\ref{th:sgd-cvx} à la \cite{lan2012optimal}]
Let us introduce $\Delta_t > 0$, such that
\begin{equation} \label{eq:delta_op}
    \|\nabla^2 f(x_t) - P_t\|_{op} \leqslant \Delta_t.
\end{equation}

\begin{align}
    \nonumber \alpha_t f(x_{t+1}) &\leqslant \alpha_t f(x_t) + \alpha_t \langle \nabla f(x_t), x_{t+1} - x_t \rangle + \frac{\alpha_t}{2} \|x_{t+1} - x_t\|_{\nabla^2 f(x_t)}^2 + \frac{\alpha_t M}{6} \|x_{t+1} - x_t\|_{T(x_t)}^3\\
    \nonumber&\leqslant \alpha_t f(x_t) + \alpha_t \langle g_t, x_{t+1} - x_t \rangle + \frac{\alpha_t}{2} \|x_{t+1} - x_t\|_{\nabla^2 f(x_t)}^2 + \frac{\alpha_t M}{6} \|x_{t+1} - x_t\|_{T(x_t)}^3\\
    \nonumber&\quad+ \alpha_t \langle \nabla f(x_t) - g_t, x_{t+1} - x_t \rangle\\
    \nonumber&\leqslant \alpha_t f(x_t) + \alpha_t \langle g_t, x_{t+1} - x_t \rangle + \frac{1}{2} \|x_{t+1} - x_t\|_{P_t}^2\\
    \nonumber&\quad+ \frac{\alpha_t - 1}{2} \|x_{t+1} - x_t\|_{\nabla^2 f(x_t)}^2 + \frac{\Delta_t}{2} \|x_{t+1} - x_t\|_{P_t}^2\\
    \nonumber&\quad+ \frac{\alpha_t M}{6} \|x_{t+1} - x_t\|_{T(x_t)}^3 + \alpha_t \langle \nabla f(x_t) - g_t, x_{t+1} - x_t \rangle\\
    \nonumber&\leqslant \alpha_t f(x_t) + \alpha_t \langle g_t, x - x_t \rangle + \frac{1}{2} \|x_t - x\|_{P_t}^2 - \frac{1}{2} \|x_{t+1} - x\|_{P_t}^2 + \frac{\alpha_t M}{6} \|x_{t+1} - x_t\|_{T(x_t)}^3\\
    \nonumber&\quad + \alpha_t \langle \nabla f(x_t) - g_t, x_{t+1} - x_t \rangle - \frac{1 - \alpha_t (1 + \Delta_t)}{2} \|x_{t+1} - x_t\|_{P_t}^2\\
    \nonumber&\leqslant \alpha_t f(x_t) + \alpha_t \langle g_t, x - x_t \rangle + \frac{1}{2} \|x_t - x\|_{P_t}^2 - \frac{1}{2} \|x_{t+1} - x\|_{P_t}^2\\
    \nonumber&\quad + \frac{\alpha_t M}{6} \|x_{t+1} - x_t\|_{T(x_t)}^3+ \frac{\alpha_t^2 \|\nabla f(x_t) - g_t\|^2}{2\mu (1 - \alpha_t (1 + \Delta_t))}\\
    \nonumber&\leqslant \alpha_t f(x_t) + \alpha_t \langle \nabla f(x_t), x - x_t \rangle + \frac{1}{2} \|x_t - x\|_{P_t}^2 - \frac{1}{2} \|x_{t+1} - x\|_{P_t}^2\\
    \nonumber&\quad + \frac{\alpha_t M}{6} \|x_{t+1} - x_t\|_{T(x_t)}^3 + \alpha_t \langle g_t - \nabla f(x_t), x - x_t \rangle + \frac{\alpha_t^2 \|\nabla f(x_t) - g_t\|^2}{2\mu (1 - \alpha_t (1 + \Delta_t))}\\
    \nonumber&\leqslant \alpha_t f(x) + \frac{1}{2} \|x_t - x\|_{P_t}^2 - \frac{1}{2} \|x_{t+1} - x\|_{P_t}^2\\
    \nonumber&\quad + \frac{\alpha_t M}{6} \|x_{t+1} - x_t\|_{T(x_t)}^3 + \alpha_t \langle g_t - \nabla f(x_t), x - x_t \rangle + \frac{\alpha_t^2 \|\nabla f(x_t) - g_t\|^2}{2\mu (1 - \alpha_t (1 + \Delta_t))}
\end{align}

Let us denote the latter two terms by the following introduced value:
\begin{align*}
    E_t := \alpha_t \langle g_t - \nabla f(x_t), x - x_t \rangle + \frac{\alpha_t^2 \|\nabla f(x_t) - g_t\|^2}{2\mu (1 - \alpha_t (1 + \Delta_t))}.
\end{align*}
\begin{align*}
    \alpha_t (f(x_{t+1}) - f(x^*)) &\leqslant \frac{1}{2} \|x_t - x^*\|_{P_t}^2 - \frac{1}{2} \|x_{t+1} - x^*\|_{P_{t+1}}^2 + \frac{\alpha_t M}{6} \|x_{t+1} - x_t\|_{T(x_t)}^3\\
    &\quad+ \frac{1}{2} \langle (P_t - P_{t+1}) (x_{t+1} - x^*), x_{t+1} - x^*\rangle + E_t
\end{align*}
Summing up for $t = 0, ..., T-1$ and taking $\mathbb{E}_{P_T}$,
\begin{align}
    \nonumber\left(\sum_{t=0}^{T-1} \alpha_t\right) \mathbb{E}_{P_T} &[f(\widetilde{x}_T) - f(x^*)] \leqslant \frac{1}{2} \|x_0 - x^*\|_{P_0}^2 - \frac{1}{2} \|x_T - x^*\|_{P_T}^2 + \sum_{t=0}^{T-1} E_t\\
    \nonumber&\quad+ \frac{M}{6} \sum_{t=0}^{T-1} \alpha_t^4 \|x_{t+1} - x_t\|_{T(x_t)}^3 + \frac{R^2}{2} \sum_{t=0}^{T-1} \|\beta_t (P_t - H_t)\|_{op}\\
    \nonumber&\stackrel{\eqref{eq:upb-cubic},\eqref{eq:alpha_conv}}{\leqslant} \frac{R^2}{2} + \sum_{t=0}^{T-1} \alpha_t \langle g_t - \nabla f(x_t), x - x_t\rangle + \frac{1}{2\mu (1 - c_1)} \sum_{t=0}^{T-1} \alpha_t^2 \|\nabla f(x_t) - g_t\|^2\\
    \nonumber&\quad+ \frac{M}{4} \sum_{t=0}^{T-1} \alpha_t^4 (\|g_t\|_{P_t}^*)^3 (1 + (\Delta_t + \sigma_T)^{3/2}) + \frac{R^2}{2} \sum_{t=0}^{T-1} \beta_t \|H_t - P_t\|_{op}\\
    \nonumber&\stackrel{\eqref{eq:alpha_conv}}{\leqslant} \frac{R^2}{2} + \sum_{t=0}^{T-1} \alpha_t \langle g_t - \nabla f(x_t), x - x_t\rangle + \frac{1}{2\mu (1 - c_1)} \sum_{t=0}^{T-1} \alpha_t^2 \|\nabla f(x_t) - g_t\|^2\\
    \label{eq:cvx-sgd-gen}&\quad+ \frac{M}{4} \sum_{t=0}^{T-1} \alpha_t^2 (\|g_t\|_{P_t}^*)^{3-2\delta} + \frac{R^2}{2} \sum_{t=0}^{T-1} \beta_t \|H_t - P_t\|_{op},
\end{align}
where it is required that $\alpha_t$ satisfies the following inequality for some $c_1 \in (0, 1)$ and $\delta \in (0, 3/2)$:
\begin{align} \label{eq:alpha_conv}
    \nonumber\alpha_t &\leqslant \frac{1}{\gamma_t} := \frac{1}{\frac{1}{c_1}(1 + \Delta_t) + (\|g_t\|_{P_t}^*)^{\delta} (1 + (\Delta_t + \sigma_T)^{3/4})}\\
    &\leqslant \min\left\{\frac{c_1}{1 + \Delta_t}, \frac{1}{(\|g_t\|_{P_t}^*)^{\delta} (1 + (\Delta_t + \sigma_T)^{3/4})}\right\}.
\end{align}

Substituting $x = x^*$ into \eqref{eq:cvx-sgd-gen}, taking $\mathbb{E}_{P_\xi}$, and using \eqref{eq:unbiasedness},
\begin{equation} \label{eq:cvx-implicit}
    \left(\sum_{t=0}^{T-1} \alpha_t\right) \mathbb{E}_{P_\xi P_T} [f(x_t) - f(x^*)] \leqslant \frac{R^2}{2} + \left(\frac{\sigma_g^2}{2 \mu (1 - c_1)} + \frac{M B^{3 - 2\delta}}{4}\right) \sum_{t=0}^{T-1} \alpha_t^2 + \frac{R^2}{2} \sum_{k=0}^{T-1} \beta_k \|H_k - P_k\|_{op},
\end{equation}
where $B = \max_{k=0, ..., T-1} \|g_t\|_{P_t}^*$. For any $t > 0$,
\begin{align}
    \nonumber\mathbb{E}_{P_\eta} [\|H_t - P_t\|_{op}] &= \mathbb{E}_{P_\eta} [\|H_t - P_{t-1} - \beta_t (H_t - P_{t-1})\|_{op}] = (1 - \beta_t) \mathbb{E}_{P_\eta} [\|H_t - P_{t-1}\|_{op}]\\
    \nonumber&\stackrel{\eqref{eq:hess-var}}{\leqslant} (1 - \beta_t) (\|H_{t-1} - P_{t-1}\|_{op} + \|\nabla^2 f(x_t) - \nabla^2 f(x_{t-1})\|_{op} + 2 \sigma_H)\\
    \nonumber&\stackrel{\eqref{eq:hess-cont}}{\leqslant} (1 - \beta_t) (\|H_{t-1} - P_{t-1}\|_{op} + M \|x_t - x_{t-1}\|_{T(x_{t-1})} + 2 \sigma_H)\\
    \nonumber&\stackrel{\eqref{eq:proj-exp}}{\leqslant} (1 - \beta_t) (\|H_{t-1} - P_{t-1}\|_{op} + M \|\alpha_{t-1} P_{t-1}^{-1} g_{t-1}\|_{T(x_{t-1})} + 2 \sigma_H)\\
    \label{eq:prec-precision}&\stackrel{\eqref{eq:upb-cubic}}{\leqslant} (1 - \beta_t) (\|H_{t-1} - P_{t-1}\|_{op} + \alpha_{t-1} M \|g_{t-1}\|_{P_{t-1}}^* (1 + \sigma_T^{1/2}) + 2 \sigma_H)
\end{align}

Taking $\mathbb{E}_{P_\eta}$ in \eqref{eq:cvx-implicit}, substituting \eqref{eq:prec-precision} into \eqref{eq:cvx-implicit}, and taking $\mathbb{E}_{P_\eta}$ again,
\begin{align}
    \nonumber\left(\sum_{k=0}^{T-1} \alpha_k\right) \mathbb{E}_{P_\eta P_\xi P_T} &[f(x_t) - f(x^*)] \leqslant \frac{R^2}{2} + \left(\frac{\sigma_g^2}{2\mu (1 - c_1)} + \frac{M B^{3 - 2\delta}}{4}\right) \sum_{k=0}^{T-1} \alpha_k^2 + \sigma_H R^2 \sum_{k=1}^{T-1} \beta_k (1 - \beta_k) \\
    \nonumber&\quad+ \frac{M R^2 (1 + \sigma_T^{1/2})}{2} \sum_{k=1}^{T-1} \beta_k (1 - \beta_k) \alpha_{k-1} \|g_{k-1}\|_{P_{k-1}}^*\\ 
    \nonumber&\quad+ \frac{R^2}{2} \sum_{k=1}^{T-1} \beta_k (1 - \beta_k) \mathbb{E}_{P_\eta}[\|H_{k-1} - P_{k-1}\|_{op}] + (1 - \beta_0) \|H_0 - P_{-1}\|_{op}\\
    \nonumber&\stackrel{\eqref{eq:alpha_conv}}{\leqslant} \frac{R^2}{2} + (1 - \beta_0) \|H_0 - P_{-1}\|_{op} + \left(\frac{\sigma_g^2}{2\mu (1 - c_1)} + \frac{M B^{3 - 2\delta}}{4}\right) \sum_{k=0}^{T-1} \alpha_k^2\\
    \nonumber&\quad+ R^2 \left(\sigma_H + \frac{M (1 + \sigma_T^{1/2})}{2}\right) \sum_{k=1}^{T-1} \beta_k (1 - \beta_k) \\
    \nonumber&\quad+ \frac{R^2}{2} \sum_{k=1}^{T-1} \beta_k (1 - \beta_k) \mathbb{E}_{P_\eta}[\|H_{k-1} - P_{k-1}\|_{op}]\\
    \nonumber&\stackrel{\eqref{eq:beta-conv},\eqref{eq:abel-test}}{\leqslant} R^2 \left(c_3 \sigma_H + \frac{c_3 M (1 + \sigma_T^{1/2}) + c_2 + 1}{2}\right) + (1 - \beta_0) \|H_0 - P_{-1}\|_{op}\\
    \label{eq:f-means}&\quad+ \left(\frac{\sigma_g^2}{2\mu (1 - c_1)} + \frac{M B^{3 - 2\delta}}{4}\right) \sum_{k=0}^{T-1} \alpha_k^2
\end{align}
where it is required that $\beta_t$ satisfies the following condition:
\begin{equation} \label{eq:beta-conv}
    \frac{1}{1 - \beta_t} \geqslant \max\left\{1 + \frac{\alpha_{t-1} M \|g_{t-1}\|_{P_{t-1}}^* (1 + \sigma_T^{1/2}) + 2 \sigma_H}{\|H_{t-1} - P_{t-1}\|_{op}}, (t+1)^2\right\}.
\end{equation}

Setting $\beta_0 = 1$ and $\delta = \frac{3}{4}$, \eqref{eq:f-means} implies
\begin{align*}
    &\mathbb{E}_{P_\eta P_\xi P_T} [f(x_t) - f(x^*)]\\ &\quad\leqslant \frac{R^2}{\alpha \sum_{k=0}^{T-1} \frac{1}{\gamma_k}} \left(c_3 \sigma_H + \frac{c_3 M (1 + \sigma_T^{1/2}) + c_2 + 1}{2}\right) + \alpha \frac{\sum_{k=0}^{T-1} \frac{1}{\gamma_k^2}}{\sum_{k=0}^{T-1} \frac{1}{\gamma_k}} \left(\frac{\sigma_g^2}{2\mu (1 - c_1)} + \frac{M B^{3 - 2\delta}}{4}\right)\\
    &\quad\stackrel{\eqref{eq:alpha_conv}}{\leqslant} \frac{R^2}{\alpha \sum_{k=0}^{T-1} \frac{1}{\gamma_k}} \left(c_3 \sigma_H + \frac{c_3 M (1 + \sigma_T^{1/2}) + c_2 + 1}{2}\right) + \alpha \frac{T}{\sum_{k=0}^{T-1} \frac{1}{\gamma_k}} \left(\frac{\sigma_g^2}{2\mu (1 - c_1)} + \frac{M B^{3 - 2\delta}}{4}\right)\\
    &\quad\stackrel{\eqref{eq:alpha-cvx}}{\leqslant} O(1) \left(\frac{R^2 (\sigma_H + M (1 + \sigma_T^{1/2}) + 1) {\displaystyle \max_{k=0, ..., T-1} \gamma_k}}{T} + \sqrt{\frac{(\sigma_g^2/\mu + M B^{3 - 2\delta}) (\sigma_H + M (1 + \sigma_T^{1/2}) + 1)}{T}} \right)\\
    &\quad\leqslant O(1) \left(1 + \sigma_H + (1 + (\sigma_H + \sigma_T)^{3/4}) B^{3/4}\right) \frac{R^2 (\sigma_H + M (1 + \sigma_T^{1/2}) + 1)}{T}\\
    &\quad\quad+ O(1) \sqrt{\frac{(\sigma_g^2/\mu + M B^{3/2}) (\sigma_H + M (1 + \sigma_T^{1/2}) + 1)}{T}},
\end{align*}
where
\begin{equation} \label{eq:alpha-cvx}
    \alpha_t = \min\left\{\frac{1}{\gamma_t}, \sqrt{\frac{2 c_3 \sigma_H + c_3 M (1 + \sigma_T^{1/2}) + c_2 + 1}{T \left(\frac{\sigma_g^2}{\mu (1 - c_1)} + \frac{M}{2}\right)}}\right\}.
\end{equation}

\end{proof}






\end{document}